\documentclass[11pt,reqno]{amsart}

\usepackage[utf8]{inputenc}
\usepackage[T1]{fontenc}
\usepackage{latexsym}
\usepackage{enumitem}
\usepackage{amsxtra}
\usepackage{amsmath}
\usepackage{amsfonts}
\usepackage{amsthm}
\usepackage{amssymb}
\usepackage{mathtools}
\usepackage{kbordermatrix}

\usepackage{mathdots}
\usepackage[noadjust]{cite}
\usepackage{comment}
\usepackage{hyperref}
\usepackage[capitalise]{cleveref}

\usepackage[centering,margin=1in]{geometry}

\newtheorem{theorem}{Theorem}[section]
\newtheorem{lemma}[theorem]{Lemma}
\newtheorem{corollary}[theorem]{Corollary}
\newtheorem{proposition}[theorem]{Proposition}

\theoremstyle{definition}
\newtheorem{definition}[theorem]{Definition}
\newtheorem{example}[theorem]{Example}
\newtheorem{notation}[theorem]{Notation}
\newtheorem{question}[theorem]{Question}

\theoremstyle{remark}

\numberwithin{equation}{section}

\newcommand{\defterm}[1]{\emph{#1}}

\DeclareMathOperator{\mr}{mr}

\DeclareMathOperator{\tri}{tri}
\DeclareMathOperator{\qual}{{\mathcal Q}}
\DeclareMathOperator{\mqual}{{\mathcal R}}
\newcommand{\X}{\mathcal X}
\newcommand{\Y}{\mathcal Y}
\newcommand{\F}{\mathbb F}

\newcommand{\setoneto}[1]{[#1]}
\newcommand{\intcl}[1]{\widehat{#1}}

\usepackage{pgf,tikz}
\tikzstyle{vertex}=[circle, draw, inner sep=0pt, minimum size=6pt]
\newcommand{\vertex}{\node[vertex]}

\definecolor{louiscolor}{named}{blue}

\DeclareMathOperator{\cl}{cl}

\setenumerate{label=\rm(\roman*),midpenalty=2}

\begin{document}
\title[Matrix patterns and matroid adjoints]{Matroid adjoints and the minimum rank of zero-nonzero matrix patterns}

\author[L.~Deaett]{Louis Deaett}
\address{Department of Mathematics and Statistics\\
  Quinnipiac University\\
Hamden, Connecticut}
\email{louis.deaett@quinnipiac.edu}
\author[K.~Grace]{Kevin Grace}
\address{Department of Mathematics and Statistics\\
University of South Alabama\\
Mobile, Alabama}
\email{kevingrace@southalabama.edu}

\subjclass{05B35, 05C50, 15B35, 15A03.}
\date{\today}

\begin{abstract}

The problem of finding the minimum rank of a matrix with a given zero-nonzero pattern has been generalized to a class of matroids associated to the pattern.  The fundamental lower bound known as the triangle number still holds in this generalized setting.  But the matroid minimum rank of a pattern need not match that of its transpose.

We associate to each pattern $\X$ a lattice $L(\X)$.
We define the \emph{fundamental pattern} of a matroid $M$ to be the complement of its hyperplane-point incidence pattern and
note that when $\X$ is the fundamental pattern of $M$, the lattice of flats of $M$ is $L(\X)$.
We then prove that, for every pattern $\X$, the dual lattice of $L(\X)$ is isomorphic to $L(\X^T)$.

We show that a matroid $M'$ of the same rank as $M$ is an adjoint of $M$ if and only if $M'$ is associated with the transpose of the fundamental pattern of $M$. Our main result ties together the notion of a matroid adjoint with the phenomenon of a gap between the triangle number $k$ and the matroid minimum rank of a pattern. Namely, we show that, if any matroid of rank $k$ associated with a pattern has an adjoint, then there is no such gap for the pattern's transpose.

We show that the matroid of minimum rank associated with the fundamental pattern is unique.  Using this, we prove that the matrix minimum rank of the fundamental pattern of a matroid over different fields depends on the representability of the matroid over those fields. This allows us to recover and improve upon a construction of Berman, Friedland, Hogben, Rothblum, and Shader (2008).  We also give a smaller example than any previously known of a pattern with a matroid minimum rank smaller than its matrix minimum rank over every field.  Finally, we establish that, for the fundamental pattern, a converse holds to our main result.  In particular, a matroid with fundamental pattern $\X$ has an adjoint if and only if the matroid minimum rank of $\X^T$ is equal to its triangle number.

\end{abstract}

\maketitle
%\linenumbers

\section{Introduction}\label{sec:introduction}

A theme of study in combinatorial matrix theory is the investigation of how some combinatorial description of a matrix reflects or reveals properties of the linear operator represented by that matrix.
One such combinatorial description is given by the \defterm{zero-nonzero pattern} of the matrix, which specifies exactly which of its entries are zero and which are nonzero.  One may then ask, among other questions, what this description implies about the rank of the matrix.  The question of the largest possible rank is well understood; it is given by the \defterm{term rank} of the pattern (see Section \ref{sec:min rank problems}).  On the other hand, the smallest possible rank, the \emph{(matrix) minimum rank} of the pattern, has a more subtle behavior, and may depend on the field within which the matrix entries are supposed to lie.

The problem of understanding how combinatorial properties of a zero-nonzero pattern relate to its matrix minimum rank has been well studied.  (See, for example, the seminal work of \cite{hs1993}, as well as \cite{johnson_and_canto,johnson_and_link,johnson_and_zhang} and \cite{bfhhhhs2009}. Also note related work in the context of sign patterns; see \cite[Section 42.6]{hog17} for a summary.)  In some work, notably \cite{berman_et_al} and \cite{d20}, a connection with the combinatorics of matroids is established and used to study this problem.  In particular, \cite{d20} introduces a parameter known as the \defterm{matroid minimum rank} of the pattern, which generalizes its matrix minimum rank by attaching to the pattern a class of matroids, one that includes and extends the class of matroids represented by the matrices with that pattern.
The goal is not just to understand this generalized minimum rank parameter for its own sake, but also to use matroid theory to better understand the behavior of the matrix minimum rank.

With that in mind, one notable property of the matroid minimum rank of a pattern is that it is not in general preserved by the transpose.  This was already noted in \cite{d20}, but no theoretical account was given there to accompany this observation.  Here, we will examine this phenomenon by establishing a connection between the behavior of the matroid minimum rank under transpose and the notion of a matroid \defterm{adjoint}.

The notion of a matroid adjoint seems to have first been introduced by Crapo \cite{c71}. However, as shown by Cheung \cite{Cheung}, not every matroid has an adjoint. 
Adjoints of matroids
received significant attention in the 1980s and 1990s in, e.g., \cite{bk86,bc88,bw89,akw90,ah95}.  More recently, there has been renewed interest in matroid adjoints, as evidenced by the recent work in \cite{fjk23,ftw24} and others.

Our present work connects adjoints of the matroids associated with a pattern (and the question of whether or not such adjoints exist) with the matter of how the matroid minimum rank is affected by taking the transpose of the pattern.
We pay particular attention to the question of when this minimum rank coincides with its most fundamental combinatorial lower bound, that of the \defterm{triangle number}.
Our main result (\cref{thm:One Big Theorem}) shows that this always happens when some matroid of the same rank associated with the transpose has an adjoint.  This suggests a strong connection between the question (studied elsewhere) of when a matroid has an adjoint and that of when a pattern has a minimum rank meeting this lower bound.

One important special case occurs when there is a unique matroid $M$ of minimum rank associated with a pattern $\X$.  We show that this is the case when $\X$ is a certain pattern encoding the incidence between the hyperplanes and points of $M$, and we illustrate how this may be used in conjunction with representability properties of $M$ to show that the matrix minimum rank of $\X$ differs over different fields.  We also have a strong connection between an adjoint of $M$ and a matroid of the same rank associated with the transpose of $\X$; Theorem \ref{thm:adjoint and transpose of hyperplane-point pattern} shows that these notions are equivalent.  This gives a converse to our main result in this context, in that $M$ has an adjoint precisely when $\X^T$ has a matroid minimum rank equal to
its triangle number.

In the remainder of this paper,
Section \ref{sec:preliminaries} presents the definitions and basic observations we need regarding matroids, matrix patterns, and lattices.
\cref{sec:min rank problems} sets out
questions on how
combinatorial properties of
a zero-nonzero pattern constrain the rank of a matrix or matroid associated to it; in particular, both the matrix and the matroid minimum rank of a pattern are defined.
Section \ref{sec:transposes and adjoints} examines how the combinatorial structures that we associate with a pattern
change under the transpose.  There we see how the phenomenon whereby the matroid minimum rank of a pattern may exceed its triangle number is connected with a failure of adjoints to exist for matroids in the class associated with its transpose.
We then introduce a construction that gives, for each matroid $M$, a pattern such that $M$ is the unique matroid of minimum rank in the class associated with that pattern.
In Section \ref{sec:uniqueness}, we prove that for this specific pattern there is an even  
stronger connection between the existence of adjoints and the effect of taking the transpose on the minimum rank.
In that section we also show that this pattern coincides with one constructed in \cite{berman_et_al}, which focused on how the matrix minimum rank of a pattern may depend on the field.
In fact, the pattern with the uniqueness properties that we exploit in Sections \ref{sec:transposes and adjoints} and \ref{sec:uniqueness} shares its purpose with, but is more universal than, a construction given in \cite{d20} that applies only to matroids of rank $3$.  How this construction could be generalized to give more powerful results is one of the open questions we set out in Section \ref{sec:questions and future directions}, where we outline some directions for future research.

\section{Preliminaries}\label{sec:preliminaries}

In this section, we recall several concepts that will be essential to our investigation in this paper.

\subsection{Matroids}

We assume the reader to be familiar with the notion of a matroid. 
However, for easy reference, we frequently cite Oxley \cite{o11} for routine or well-known matroid results. Additionally, we follow \cite{o11} for matroid notation and terminology not defined in this paper.

In terms of matroids, the notion most essential for our purposes is that of a \defterm{flat}, a set that is rank-maximal, i.e., such that adding any additional element to the set would increase its rank.
A \emph{hyperplane} of a matroid is a flat whose rank is $1$ less than the rank of the matroid.

\subsection{Matrix patterns}

In combinatorial matrix theory, it is common to consider some description of a matrix that captures only the signs of its entries, or perhaps only the location of its nonzero entries.  Such a description is, in general, known as a \defterm{matrix pattern}.  The case of matrix \defterm{sign patterns} is well studied (see \cite[Chapter 42]{hog17}) and has a rich history of applications.  The closely-related notion of a matrix pattern we focus on here has also received a good deal of attention, and begins with the following.

\begin{definition}\label{def:zero-nonzero pattern}
    A \defterm{zero-nonzero pattern} (or simply a
    \defterm{pattern} for short) is a matrix with entries from the set $\{0,*\}$. In particular, if $A$ is any matrix (with entries from some set that includes $0$), then the \defterm{zero-nonzero pattern of $A$} is obtained from $A$ by replacing each nonzero entry with $*$.
\end{definition}

\begin{example}\label{ex:Example 25 pattern}
An important example we consider in this paper is the zero-nonzero pattern given as \cite[Example 25]{d20}, namely
\begin{equation}\label{eqn:ex 25 pattern}
\X = \begin{bmatrix}
%       1   2   3   4   5   6   7
		* & 0 & * & 0 & * & 0 & * \\  %  246    -4
		* & 0 & 0 & * & 0 & * & * \\  %  235    -2
		* & * & 0 & 0 & * & * & 0 \\  %  347    -6
		* & * & * & * & 0 & 0 & 0 \\  %  567    -7
		* & * & * & * & 0 & 0 & * \\  %   56    -8
		0 & * & 0 & 0 & * & * & 0 \\  % 1347    -5
		0 & 0 & * & 0 & * & 0 & * \\  % 1246    -3
		0 & 0 & 0 & * & 0 & * & * \\  % 1235    -1
\end{bmatrix}.
\end{equation}
\end{example}

\begin{definition}
Let $\X$ be a zero-nonzero pattern whose set of column labels is $C$ and whose set of row labels is $R$, and let $x\in R$. The \emph{support} of $x$ is the subset of $C$ consisting of the labels of those columns having a nonzero entry in row $x$.
Denote the support of $x$ by $S(x)$. Then $Z(x)=C-S(x)$ is the \emph{zero set} of $x$. Similarly, if $x\in C$, then the \emph{support} of $x$, denoted by $S(x)$, is the subset of $R$ 
consisting of the labels of those rows having a nonzero entry in column $x$, and $Z(x)=R-S(x)$ is the \emph{zero set} of $x$.
\end{definition}

We extend the definition of the zero set as follows.

\begin{definition}
\label{def:zeroset}
Let $\X$ be a zero-nonzero pattern. If $A$ is a set either of column labels or of row labels of $\X$, then the \emph{zero set} of $A$, denoted $Z(A)$, is $\bigcap\limits_{a\in A}Z(a)$.
\end{definition}

\subsection{Lattices}
Another notion that will be central to what follows is that of a lattice.  (The notion of ``lattice'' here is that of order theory; see \cite{crawley_dilworth} for a comprehensive reference on this topic.)

\begin{definition}
    A \defterm{lattice} is a partially ordered set $L$ such that, for every $x,y\in L$, there exists a unique least upper bound (called the \defterm{join}) of $x$ and $y$ and a unique greatest lower bound (called the \defterm{meet}) of $x$ and $y$.
\end{definition}

The flats of a matroid $M$, when ordered by inclusion, form a lattice, called the \emph{lattice of flats} of $M$. We denote this lattice by $L(M)$.

\begin{definition}
The \defterm{dual} $L^*$ of a lattice $L$ is obtained by reversing all order relations. (In 
terms of Hasse diagrams, $L^*$ is obtained by ``turning $L$ upside down.'') 
\end{definition}

Importantly, lattice duality does not coincide with matroid duality, in the sense that, in general, $(L(M))^*\neq L(M^*)$.

\begin{definition}
\label{def:intersection-lattice}
Let $E$ be a finite set and let $\mathcal{C}$ be a collection of subsets of $E$. The \emph{intersection lattice} generated by $\mathcal{C}$, denoted by $L(\mathcal{C})$, is the lattice, ordered by inclusion, whose elements are all intersections of sets in $\mathcal{C}$ (including $E$ as the ``empty intersection'').
\end{definition}

It is not difficult to see that $L(\mathcal{C})$ is indeed a lattice. If $X,Y\in L(\mathcal{C})$, then the meet of $X$ and $Y$ is $X\cap Y$. The join of $X$ and $Y$ also must exist, since we may collect all of the sets in $L(\mathcal{C})$ that contain both $X$ and $Y$, and then the intersection of all of these is bound to be in $L(\mathcal{C})$ as well.

\begin{definition}\label{def:lattice of pattern}
Let $\X$ be a zero-nonzero pattern with $n$ columns. The \defterm{intersection lattice} associated with $\X$, denoted by $L(\X)$, is the intersection lattice generated by $\mathcal{C}$, where $\mathcal{C}$ is the collection of subsets of $\setoneto{n}$ consisting of the zero sets of the rows of $\X$.
\end{definition}

\subsection{The fundamental pattern of a matroid}
\label{subsection:complement}

The \emph{hyperplane-point incidence pattern} of a matroid is the zero-nonzero pattern with a row for each hyperplane and a column for each point, in which the $(i,j)$-entry is $*$ if and only if hyperplane $i$ contains point $j$. Therefore, the \emph{complement} of the hyperplane-point incidence pattern of a matroid is the zero-nonzero pattern $\X$ with a row for each hyperplane and a column for each point such that the $(i,j)$-entry of $\X$ is $0$ if and only if hyperplane $i$ contains point $j$. The significance of this pattern leads us to define the following term.

\begin{definition}
    \label{def:fundamental}
    The \emph{fundamental pattern of a matroid} is the complement of its hyperplane-point incidence pattern. 
\end{definition}

Note that the collection of hyperplanes of a matroid completely determines the matroid. In particular, every flat of a matroid is an intersection of some of its hyperplanes. This leads to the following result. 

\begin{lemma}
    \label{lem:CHPIP generates lattice of flats}
    Let $\X$ be the fundamental pattern of a matroid $M$. Then $L(\X)=L(M)$.
\end{lemma}

\begin{proof}
    Let $[n]$ be the ground set of $M$. Then $[n]$ is also the set of column labels of $\X$. A subset of $[n]$ is a member of $L(\X)$ if and only if it is the intersection of the zero sets of some collection of rows of $\X$. The zero sets of the rows of $\X$ are precisely the hyperplanes of $M$. Thus, a subset of $[n]$ is a member of $L(\X)$ if and only if it is the intersection of some collection of hyperplanes of $M$.

    Now, every hyperplane of a matroid is also a flat, and every intersection of flats is also a flat \cite[Lemma 1.7.3]{o11}. Therefore, every intersection of hyperplanes is a flat. Conversely, by \cite[Proposition 1.7.8]{o11}, every flat is an intersection of hyperplanes. It follows that a subset of $[n]$ is a member of $L(M)$ if and only if it is the intersection of some collection of hyperplanes of $M$. Therefore, $L(\X)=L(M)$. 
\end{proof}

\section{Minimum rank problems}\label{sec:min rank problems}

We now survey two problems concerned with determining how combinatorial properties of a zero-nonzero pattern reflect or constrain the rank of an object associated with that pattern.  As noted in Section \ref{sec:introduction}, the most interesting such questions concern how small the rank can be under such constraints.  So the following notation will be helpful.

\begin{notation}
Given a collection $\mathcal C$ of objects, each having some nonnegative integer rank, we write $\mr \mathcal C$ to denote the smallest rank of an object in $\mathcal C$.
\end{notation}

\subsection{The minimum rank problem for matrices}

Of course, the most natural object with combinatorics that can be captured by a zero-nonzero pattern is a matrix.

\begin{definition}
Let $\X$ be a zero-nonzero pattern and $\mathbb F$ be a field.  We let $\qual_{\mathbb F}(\X)$ denote the \defterm{qualitative class of matrices over $\mathbb F$ with pattern $\X$}, which is simply the set of all matrices with the pattern $\X$ and entries in $\mathbb F$.
In a context in which the choice of field $\mathbb F$ is implied, or irrelevant, we may  simply write $\qual(\X)$.
\end{definition}

\begin{definition}
Let $\X$ be a zero-nonzero pattern and $\mathbb F$ be a field. The \defterm{matrix minimum rank of $\X$ over $\mathbb F$} is the smallest rank of a matrix over $\mathbb F$ that has the pattern $\X$, i.e., it is the value of $\mr\qual_{\mathbb F}(\X)$.
\end{definition}

The \defterm{matrix minimum rank problem} for zero-nonzero patterns is that of describing how combinatorial properties of a pattern $\X$ constrain its matrix minimum rank.  In fact, the more general problem of determining all values possible for the rank of a matrix in $\qual_{\F}(\X)$ is not actually more substantial, which may be seen following two observations.  First, when $\mathbb F$ has at least $3$ elements, the maximum rank of a matrix in $\qual_{\F}(\X)$ is the \defterm{term rank} of $\X$,
a combinatorial parameter that is well understood \cite{br91}.
It is defined as follows.

\begin{definition}\label{def:term rank}
Let $\X$ be a zero-nonzero pattern.  The maximum number of $*$ entries that can be chosen from $\X$ such that no two of the entries lie within the same row or column is the \defterm{term rank} of $\X$.
\end{definition}

Second, it is easy to see that every rank in between the minimum and the maximum can be achieved; in particular, one may start with a matrix achieving the minimum rank and then, by changing just one entry at a time, produce a sequence of matrices that ends with one achieving the maximum rank.  Since modifying a single entry in a matrix can change its rank by at most $1$, it follows that this sequence will include a matrix of every rank from the minimum to the maximum.

\subsection{The minimum rank problem for matroids}

Following \cite{d20}, we may also associate to each zero-nonzero pattern a class of matroids.

\begin{definition}\label{def:matroid qualitative class}
Let $\X$ be a zero-nonzero pattern with $n$ columns.  We let $\mqual(\X)$ denote the set containing every matroid $M$ on ground set $\setoneto{n}$ with the property that the zero set of each row of $\X$ is a flat of $M$.
\end{definition}

Note that the class $\mqual(\X)$
is never empty; for example, it contains the free matroid on $\setoneto{n}$.  
In addition, an abundant class of matroids residing in $\mqual(\X)$ is given by the following crucial result.

\begin{theorem}[{\cite[Theorem 12]{d20}}]\label{thm:matrix in Q represents matroid in R}
    For each matrix $A$ in $\qual(\X)$, the matroid $M[A]$ represented by $A$ is in $\mqual(\X)$.
\end{theorem}

\cref{thm:matrix in Q represents matroid in R} shows that the class $\mqual(\X)$ is a generalization of the class $\qual(\X)$ in the sense that  $\mqual(\X)$ contains every matroid represented by a matrix in $\qual(\X)$, and potentially some additional matroids as well.

\begin{definition}
Let $\X$ be a zero-nonzero pattern. We refer to $\mr\mqual(\X)$ as the \defterm{matroid minimum rank of $\X$}.
\end{definition}

The \defterm{matroid minimum rank problem} for zero-nonzero patterns is that of describing how combinatorial properties of a pattern $\X$ constrain its matroid minimum rank.
In a parallel with the matrix minimum rank problem, the general question of what values are possible for the rank of a matroid in $\mqual(\X)$ again reduces to determining the minimum.
This is true since, as noted above, the free matroid is in $\mqual(\X)$ and shows that the maximum rank is equal to the number of columns of $\X$, while every rank in between the minimum and maximum is achieved by some matroid, which we will now prove.
As with the matrices in $\qual(\X)$, the key observation is that we can begin with an object of minimum rank and transform it into one of maximum rank by way of a series of steps, each of which can change the rank by at most $1$.  

First, given a zero-nonzero pattern $\X$ with $n$ columns, let $\X_0=\X$.
Denote by $\mathcal I_n$ the zero-nonzero pattern of the $n\times n$ identity matrix.
Then, for each $i$ with $1 \le i \le n$,  let $\X_i$ be the result of appending to $\X_0$ the first $i$ rows of $\mathcal I_n$.  Finally, for each $i$ with $0 \le i \le n$, take $M_i$ to be any matroid of minimum rank in $\mqual(\X_i)$.
It is clear from \cref{def:matroid qualitative class} that
\[
\mqual(\X) = \mqual(\X_0) \supseteq \mqual(\X_1) \supseteq \cdots \supseteq \mqual(\X_n),
\]
so that, in particular, each $M_i$ is in $\mqual(\X)$.  Note that $r(M_0)$ is the matroid minimum rank of $\X$.  Also note that,
since $L(\mathcal I_n)$ includes every subset of $\setoneto{n}$, the free matroid on $\setoneto{n}$ is the only matroid in $\mqual(\mathcal I_n) \supseteq \mqual(\X_n)$.  Hence, $M_n$ must be this matroid, and so 
$r(M_n)=n$.  It therefore remains to show only that the ranks of two matroids adjacent in the sequence $M_0,M_1,\ldots,M_n$ can differ by at most $1$.  This follows from the fact that appending any one additional row to a zero-nonzero pattern cannot change its matroid minimum rank by more than $1$, a fact that we prove as the following theorem.

\begin{theorem}
\label{thm:append_row}
Let $\X$ be a zero-nonzero pattern, and let $\X'$ be obtained from $\X$ by appending one row. Then $\mr\mqual(\X)\leq\mr\mqual(\X')\leq\mr\mqual(\X)+1$.
\end{theorem}

\begin{proof}
    Let $n$ be the number of columns of $\X$. It is well known (see, e.g., \cite[Exercise 3.25]{gm12}) that, for a matroid $M$ on ground set $\setoneto{n}$, a set $F$ is a flat of $M$ if and only if $\setoneto{n}\setminus F$ is a cyclic set (union of circuits) of $M^*$. Hence, letting $\mathcal{R^*}(\X)$ denote the set of matroids such that the support of each row of $\X$ is a cyclic set of the matroid, we have $M\in \mqual(\X)$ if and only if $M^* \in \mqual^*(\X)$.
    
    Now let $k$ and $k'$ be the maximum rank among all matroids in $\mathcal{R^*}(\X)$ and in $\mathcal{R^*}(\X')$, respectively.  It suffices to show that $k\geq k'\geq k-1$.  The fact that $\mathcal{R^*}(\X')\subseteq\mathcal{R^*}(\X)$ immediately gives $k'\leq k$.  So, to complete the proof, we need only show $k-1 \leq k'$, for which it suffices to produce a matroid of rank at least $k-1$ in $\mqual^*(\X')$.
    
	Toward this end, let $M$ be a matroid in $\mathcal{R^*}(\X)$ with rank $k$. Let $T$ be the support of the row that was appended to $\X$ to produce $\X'$. If $T$ is cyclic in $M$, then $M\in\mathcal{R^*}(\X')$, in which case $M$ is the desired matroid
    
    On the other hand, if $T$ is not cyclic in $M$, let $C$ be the maximal subset of $T$ that is cyclic in $M$. Then $T-C$ is independent and nonempty.  In particular, $T-C$ cannot have rank $0$.  Therefore, letting $\mathcal{M}$ be the modular cut of $M$ consisting of all flats containing $T-C$, we have that $\mathcal M$ is not the set of all flats.  Hence, we may let $M'$ be the elementary quotient of $M$ with respect to $\mathcal{M}$, meaning (see \cite[Section 7.3]{o11}) that $M'=(M+_{\mathcal{M}}e)/e$, where $M+_{\mathcal{M}}e$ is the matroid obtained from $M$ by freely adding a point $e$ to the closure of $T-C$.  Then, by \cite[Theorem 7.2.3 and Proposition 3.1.6]{o11}, we have $r(M')=k-1$.
    
    Now, \cite[Proposition 7.3.6(iv)]{o11} implies that $C$ is cyclic in $M'$. Meanwhile, since $T-C$ is independent in $M$, it follows from \cite[Theorem 7.2.3]{o11} that $(T-C)\cup\{e\}$ is a circuit of $M+_{\mathcal{M}}e$, from which it follows that $T-C$ is a circuit in $M'$. So now $T$ is cyclic in $M'$. Hence, $M'\in\mathcal{R^*}(\X')$, and in this case $M'$ is the desired matroid.
\end{proof}

It is worth noting that 
questions about the matroids in $\mqual(\X)$ can be posed entirely apart from any discussion of matrices, since the statement that $M\in\mqual(\X)$ simply means that every one of some given collection of sets (namely the zero sets of the rows of $\X$) is a flat of $M$.  In particular, the matroid minimum rank problem is equivalent to the problem of determining how small the rank of a matroid on some fixed ground set can be, given only that certain subsets are flats.

Finally, we note a crucial inequality that follows immediately from \cref{thm:matrix in Q represents matroid in R}, namely that, over every field $\F$,
\begin{equation}\label{eqn:matroid min rank leq matrix min rank}
    \mr\mqual(\X) \le \mr\qual_{\F}(\X).
\end{equation}

\subsection{The triangle number lower bound}
\label{subsec:triangle number}

A combinatorial invariant of a zero-nonzero pattern that is of fundamental importance is its \defterm{triangle number}, which we now define.
\begin{definition}
    A square zero-nonzero pattern $\X$ is called a \defterm{triangle} if it is possible to permute the rows and columns of $\X$ (independently) to obtain a pattern that is lower triangular with only $\ast$ entries on its diagonal.  (Equivalently, the set system given by the supports of the rows of $\X$ has a unique transversal.)
\end{definition}

\begin{definition}
    Let $\X$ be a zero-nonzero pattern.  The \defterm{triangle number} of $\X$ is the largest $k$ such that some $k\times k$ submatrix of $\X$ is a triangle.  This value is denoted by $\tri(\X)$.
\end{definition}

An argument involving elementary linear algebra is sufficient to show that $\tri(\X)$ is a lower bound on the rank of every matrix with the pattern $\X$.  In fact, a stronger result holds in the generalized setting of the matroid minimum rank.  In particular, \cite[Lemma 24]{d20} states that the triangle number of $\X$ is the height of the lattice $L(\X)$.  Every set in $L(\X)$ is a flat of each matroid $M \in \mqual(\X)$, and it follows that $L(\X)$ is a subposet of $L(M)$.  Hence, the height of $L(M)$, which is $r(M)$, must be at least the height of $L(\X)$, namely $\tri(\X)$.  Therefore, we have $\tri(\X) \le \mr\mqual(\X)$.
Combining this with \eqref{eqn:matroid min rank leq matrix min rank} gives that, over every field $\F$,
\begin{equation}\label{eqn:fundamental triangle number inequality}
    \tri(\X) \le \mr\mqual(\X) \le \mr\qual_{\F}(\X).
\end{equation}

In the context of the matrix minimum rank problem, a key question is that of when the minimum rank of a pattern actually coincides with its triangle number.  Equivalently, understanding when there is a gap between the two, such that $\tri(\X) < \mr\qual_{\F}(\X)$, is an active area of investigation.  In light of \eqref{eqn:fundamental triangle number inequality}, it is therefore interesting to ask the analogous question for the matroid minimum rank.  That is, what conditions are necessary or sufficient to ensure or preclude that $\tri(\X) < \mr\mqual(\X)$ holds?  Many of our theorems that follow bear on this question, and specifically relate it to the existence of matroid adjoints.

\section{Transposes and adjoints}\label{sec:transposes and adjoints}

This section is the heart of the paper.  We introduce the notion of a matroid adjoint and relate it to the transpose of a matrix pattern, with the key observation that transposing the pattern has the effect of dualizing the associated intersection lattice.  
\cref{thm:One Big Theorem}, our main result, brings these ideas together to explain a gap between the triangle number and the matroid minimum rank in terms of matroid adjoints.

\subsection{The transpose and the minimum rank of a pattern}\label{subsec:transpose and min rank}

Letting $\X$ be the pattern of Example \ref{ex:Example 25 pattern},
we have, as noted in \cite[Example 25]{d20}, that $\mr\mqual(\X)\geq 5 > 4 = \tri(\X)$.
Given that the uniform matroid $U_{5,7}\in\mqual(\X)$, we in fact have $\mr\mqual(\X)=5$. Nevertheless, for the pattern
\begin{equation}\label{eqn:ex 25 transpose}
\X^T = \kbordermatrix{
& 1 & 2& 3& 4& 5& 6 & 7& 8\\
&* & * & * & * & * & 0 & 0 & 0 \\
&0 & 0 & * & * & * & * & 0 & 0 \\
&* & 0 & 0 & * & * & 0 & * & 0 \\
&0 & * & 0 & * & * & 0 & 0 & * \\
&* & 0 & * & 0 & 0 & * & * & 0 \\
&0 & * & * & 0 & 0 & * & 0 & * \\
&* & * & 0 & 0 & * & 0 & * & *
}
\end{equation}
we have $\mr\mqual(\X^T)=4$, so that $\X$ shows a discrepancy between its matroid minimum rank and that of its transpose.  One matroid in $\mqual(\X^T)$ realizing the minimum rank of $4$ is the V\'amos matroid.  In fact, with this matroid labeled as shown in Figure \ref{fig:Vamos}, the zero set of each row of $\X^T$ is not only a flat, but is actually a hyperplane.

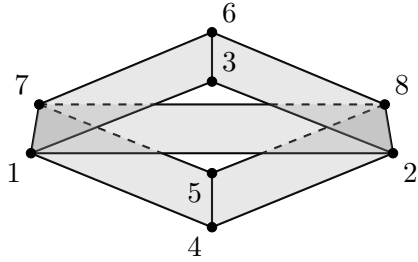
\begin{figure}
\begin{tikzpicture}
	\def \sf {0.0625};
	\def \vadjust {2};
%	\node (pict) at (42.175,28) {\includegraphics[scale=5*\sf]{ideal-pic}};
%	\draw (0,0) grid (83,53);
	
	\coordinate (C1) at (41*\sf,14*\sf + 2.5*\sf);
	\coordinate (C2) at (41*\sf,3*\sf + \vadjust*\sf);
	\coordinate (D1) at (41*\sf,48.5*\sf - \vadjust*\sf);
	\coordinate (D2) at (41*\sf,39*\sf - 3*\sf);
	\coordinate (A1) at (4.25*\sf,31.125*\sf);
	\coordinate (A2) at (2.5*\sf,20.75*\sf);
	\coordinate (B1) at (77.75*\sf,31.125*\sf);
	\coordinate (B2) at (79.5*\sf,20.75*\sf);
	
	\coordinate (X1) at (intersection of A1--B1 and A2--D2);
	\coordinate (X2) at (intersection of A1--B1 and B2--D2);
	\coordinate (X3) at (intersection of A1--C1 and A2--D2);
	\coordinate (X4) at (intersection of B1--C1 and B2--D2);
	\coordinate (X5) at (intersection of A1--C1 and A2--B2);
	\coordinate (X6) at (intersection of B1--C1 and B2--A2);
	
	\draw[thick] (D1) -- (D2);
	\draw[thick] (C1) -- (C2);

	\draw[thick,dashed] (A1) -- (X1);
	\draw[thick]        (X1) -- (X2);
	\draw[thick,dashed] (X2) -- (B1);

	\draw[thick] (A2) -- (B2);

	\draw[thick] (A1) -- (A2);
	\draw[thick] (A1) -- (D1);
	\draw[thick] (A2) -- (D2);
	\draw[thick] (A2) -- (C2);
	\draw[dashed,thick] (A1) -- (X5);
	\draw[thick] (X5) -- (C1);

	\draw[thick] (B1) -- (B2);
	\draw[thick] (B1) -- (D1);
	\draw[thick] (B2) -- (D2);
	\draw[thick] (B2) -- (C2);
	\draw[dashed,thick] (B1) -- (X6);
	\draw[thick] (X6) -- (C1);

    \fill[gray!90,fill opacity=0.2] (A1) -- (B1) -- (B2) -- (A2) -- (A1);
    \fill[gray!90,fill opacity=0.2] (A1) -- (D1) -- (D2) -- (A2) -- (A1);
    \fill[gray!90,fill opacity=0.2] (A1) -- (C1) -- (C2) -- (A2) -- (A1);
    \fill[gray!90,fill opacity=0.2] (B1) -- (D1) -- (D2) -- (B2) -- (B1);
    \fill[gray!90,fill opacity=0.2] (B1) -- (C1) -- (C2) -- (B2) -- (B1);

	\foreach \s in {A1,A2,B1,B2,C1,C2,D1,D2}
		\draw[fill] (\s) circle (\sf);
	
    \draw (A1) node[above left] {$7$};
    \draw (A2) node[below left] {$1$};
    \draw (B1) node[above right] {$8$};
    \draw (B2) node[below right] {$2$};
    \draw (C1) node[below left] {$5$};
    \draw (C2) node[below left] {$4$};
    \draw (D1) node[above right] {$6$};
    \draw (D2) node[above right] {$3$};

\end{tikzpicture}
\caption{The V\'amos matroid}
\label{fig:Vamos}
\end{figure}

Given that the rows of $\X^T$ reflect a specific selection
of hyperplanes of the V\'amos matroid,
it is natural to ask whether any special property of that matroid accounts for the fact that the minimum rank of the pattern differs from that of its transpose.  We will see in this section how this is in fact the case.  Certainly the V\'amos matroid has some remarkable properties; it is a smallest matroid that is not representable over any field, for example.  The relevant property here, however, turns out to be its failure to have an adjoint.

\subsection{Matroid adjoints}
Recall that we write $L(M)$ to denote the lattice of flats of a matroid $M$.
The notion of a matroid adjoint represents an attempt to attach a matroid to the lattice-theoretic dual $L(M)^*$ of $L(M)$. (The dual lattice $L(M)^*$ is already the lattice of flats of some matroid $M'$ exactly when $M$ is a modular matroid. In that case, $M'$ is the unique adjoint of $M$. See Section 4, particularly Proposition 4.1 and Theorem 4.3, of \cite{ftw24}.)

\begin{definition}
\label{def:adjoint}
    Let $M$ be a matroid.  A simple matroid $M'$ is said to be an \defterm{adjoint} of $M$ if $r(M')=r(M)$ and there is a map $\phi:L(M)\rightarrow L(M')$ such that
    \begin{enumerate}
    \item $\phi$ is injective. 
    \item\label{prop:order-reversing} If $F_1,F_2\in L(M)$ and $F_1\subseteq F_2$, then $\phi(F_2)\subseteq \phi(F_1)$.
    \item The restriction of $\phi$ that maps hyperplanes of $M$ to points of $M'$ is bijective.
    \end{enumerate}
\end{definition}

In light of property \ref{prop:order-reversing} of \cref{def:adjoint}, the map $\phi$ is seen to be an order-embedding of $L(M)^*$ into the lattice of flats of $M'$.

The fact that not every matroid has an adjoint was first shown by Cheung, who established in \cite{Cheung} that the V\'amos matroid does not.

\subsection{Lattice-theoretic duality and the transpose}

The notion of the dual lattice (of the lattice of flats of a matroid) is elemental to the definition of the matroid adjoint.  So, to explain what the notion of an adjoint has to do with the relationship between the minimum rank of a pattern and that of its transpose, we begin with \cref{thm:transpose}, which shows that
taking the transpose of a pattern has the effect of dualizing its intersection lattice.
To prove this, we will need the definition below and the lemma that follows.

For the following definition, we recall \cref{def:zeroset}.

\begin{definition}
\label{def:neighborhood}
    Let $\X$ be a zero-nonzero pattern whose set of column labels is $C$, and let $A\subseteq C$. The \emph{neighborhood} $N(A)$ of $A$ is $\{t\in C: Z(A)\subseteq Z(t)\}$. Similarly, if $A\subseteq R$, then the \emph{neighborhood} $N(A)$ of $A$ is $\{t\in R: Z(A)\subseteq Z(t)\}$. For simplicity, we write $N(\{x\})$ as $N(x)$ for each $x\in C\cup R$.
\end{definition}

The neighborhood function $N$ can be thought of as a ``closure operator'' that adds to a set $A$ all rows or columns possible such that $A$ and $N(A)$ have the same zero set (see \cref{lem:zero sets lemma}\ref{property:Z of A is Z of N of A}). Since the term \emph{closure} has an established meaning in matroid theory, we have opted to use a different term. However, the term \emph{closure} might be just as good of a term as \emph{neighborhood}, because the notion of neighborhood we have defined generalizes the matroid-theoretic notion of closure, as we will see in Proposition \ref{prop:closure}.

\begin{lemma}\label{lem:zero sets lemma}
Let $\X$ be a zero-nonzero pattern 
whose set of row labels is $R$ and whose set of column labels is $C$, and suppose that either $A\subseteq R$ or $A\subseteq C$. Then:
\begin{enumerate}
    \item\label{property:Z of A is Z of N of A}
    $Z(N(A))=Z(A)$.
    \item\label{property:Z of Z of A is N of A}
    $Z(Z(A))=N(A)$.
\end{enumerate}
In addition, when $A \subseteq C$, we have:
\begin{enumerate}[resume]
    \item\label{property:A is least in LX containing A}
    $N(A)$ is the least element of $L(\X)$ that contains $A$.
    \item\label{property:Z of Z of A is A}
    If $A \in L(\X)$, then $Z(Z(A))=A$.
\end{enumerate}
\end{lemma}

\begin{proof}
    Claim \ref{property:Z of A is Z of N of A} follows immediately from \cref{def:zeroset,def:neighborhood}.

To prove \ref{property:Z of Z of A is N of A}, we assume for convenience that $A \subseteq C$; the case where $A \subseteq R$ follows similarly.  By definition, $c \in N(A)$ means that $Z(A)\subseteq Z(c)$, which in turn means that column $c$ has an entry of $0$ in every row that is in $Z(A)$.  But it is obviously equivalent to say that every row in $Z(A)$ has an entry of $0$ in column $c$, i.e., to say that $c \in \bigcap_{r \in Z(A)} Z(r) = Z(Z(A))$.

    Now, to prove \ref{property:A is least in LX containing A} and \ref{property:Z of Z of A is A}, we assume for the remainder of the proof that $A \subseteq C$.

For \ref{property:A is least in LX containing A}, let $\hat A$ be the least element of $L(\X)$ that contains $A$.
Recall that $L(\X)$ is generated by closing the collection of zero sets of rows of $\X$ under intersection.
Hence,
$\hat A$ is the intersection of the zero sets of all rows of $\X$ whose zero sets contain $A$.
Let $R'$ be the collection of all such rows.  Then, to restate the above, $\hat A = Z(R')$.

For each $r \in R$, we have $r \in Z(A)$ exactly when every column in $A$ has a $0$ entry in row $r$.  This is the same as saying that row $r$ has a $0$ entry in every column of $A$, which is clearly equivalent to saying that the zero set of row $r$ contains $A$, i.e., that $r \in R'$.  Hence, $R'=Z(A)$.

Now, by the above, and making use of \ref{property:Z of Z of A is N of A}, we have $\hat A = Z(R') = Z(Z(A)) = N(A)$.

    Finally, we have from \ref{property:Z of A is Z of N of A} that $Z(Z(A))=N(A)$.  If $A\in L(\X)$, then \ref{property:A is least in LX containing A} gives 
    $N(A)=A$. So this suffices to establish \ref{property:Z of Z of A is A}.
\end{proof}

As mentioned earlier, the closure operator
of a matroid $M$, denoted $\cl_M$, is generalized by the notion of neighborhood given in \cref{def:neighborhood}.  In particular, the two coincide for the fundamental pattern of $M$, as we now show.

\begin{proposition}
\label{prop:closure}
Let $\X$ be the fundamental pattern of a matroid $M$, and let $C$ be the set of column labels of $\X$, so that $E(M)=C$.  If $A\subseteq C$, then $N(A)=\textnormal{cl}_M(A)$.
\end{proposition}

\begin{proof}
    It is well known that the closure of a set $A$ in a matroid $M$ is the smallest flat, i.e., the smallest set in $L(M)$, containing $A$.
Hence, given that $L(M)=L(\X)$ by \cref{lem:CHPIP generates lattice of flats}, the result follows from \cref{lem:zero sets lemma}\ref{property:A is least in LX containing A}.
\end{proof}

With \cref{lem:zero sets lemma} in hand, we can now prove the following result.

\begin{theorem}
\label{thm:transpose}
Let $\X$ be a zero-nonzero pattern.
Then $ L(\X)^*\simeq L(\X^T)$, with an isomorphism $\phi:L(\X)^*\rightarrow L(\X^T)$ given by $\phi(A)=Z(A)$.
\end{theorem}

\begin{proof}
First, we show that $\phi$ is well-defined. Suppose $A\in L(\X)^*$. Then, since dual lattices are posets on the same set, $A\in L(\X)$. So $A$ is a set of columns of $\X$, and thus a set of rows of $\X^T$. Hence, $Z(A)$ is the intersection of the zero sets of some set of rows of $\X^T$. Therefore, $Z(A)\in L(\X^T)$, and $\phi$ is well-defined.

Now we show the surjectivity of $\phi$. Let $T\in L(\X^T)$.  Then $T$ is a set of columns of $\X^T$ and therefore a set of rows of $\X$. Because $L(\X)$ is the intersection lattice of the zero sets of the rows of $\X$, there is a set $A\in L(\X)$ such that $A=Z(T)$. Therefore, $Z(A)=Z(Z(T))=T$ by \cref{lem:zero sets lemma}\ref{property:Z of Z of A is A}, with $\X^T$ playing the role of $\X$. Since $A\in L(\X)^*$ and $Z(A)=T$, the surjectivity of $\phi$ follows.

To see that $\phi$ is injective, let $A,B\in L(\X)^*$. Then $A,B\in L(\X)$. If $\phi(A)=\phi(B)$, then $Z(A)=Z(B)$. Thus, $Z(Z(A))=Z(Z(B))$. By \cref{lem:zero sets lemma}\ref{property:Z of Z of A is A}, we have $A=B$.

Finally, we claim that, if $A,B\in L(\X)^*$, then $A\subseteq B$ (meaning that $B\leq A$ in $L(\X)^*$) if and only if $Z(B)\subseteq Z(A)$. It is clear that if $A\subseteq B$, then $Z(B)\subseteq Z(A)$. Similarly, if $Z(B)\subseteq Z(A)$, then $Z(Z(A))\subseteq Z(Z(B))$. Then, since $A,B\in L(\X)$, we have by \cref{lem:zero sets lemma}\ref{property:Z of Z of A is A} that $A\subseteq B$.

Therefore, $\phi$ is an isomorphism of lattices.
\end{proof}

\subsection{Minimum rank and adjoints}
Our next goal is to show how the presence of a gap between the matroid minimum rank of a pattern $\X$ and the triangle number of $\X$ is related to the nonexistence of adjoints for the matroids associated with $\X^T$, i.e., those in $\mqual(\X^T)$.
Toward that goal, we now prove several results leading to our main result of \cref{thm:One Big Theorem}.

\begin{theorem}\label{thm:adjoint and transpose of hyperplane-point pattern}
Let $\X$ be the fundamental pattern of a matroid $M$. Then a matroid $M'$ is an adjoint of $M$ if and only if $M'\in\mqual(\X^T)$ and $r(M')=r(M)$. 
\end{theorem}

\begin{proof}
  First, suppose that $M'$ is an adjoint of $M$. By \cref{def:adjoint}, the ranks of $M$ and $M'$ are equal, and there is an injective, inclusion-reversing map $\phi:L(M)\rightarrow L(M')$ that maps the hyperplanes of $M$ bijectively onto the points of $M'$.

  Let $\{p_1,p_2,\dots,p_n\}$ be the set of points (rank-$1$ flats) of $M$, and let $\{H_1,H_2,\dots, H_m\}$ be the set of hyperplanes of $M$. Then, as noted above, $\phi(H_1),\phi(H_2),\dots, \phi(H_m)$ are the points of $M'$.  Now, $M'\in\mqual(\X^T)$ means that each row of $\X^T$ is the complement of the incidence pattern of some flat of $M'$. 
  We claim that, in particular, the zero set of the $i$th row of $\X^T$ is the flat $\phi(p_i)$ of $M'$.
  For this, it suffices to show that each $(i,j)$-entry of $\X^T$ is $0$ if and only if $\phi(H_j)\subseteq\phi(p_i)$.

To this end, note that the $(i,j)$-entry of $\X^T$ is $0$ if and only if the $(j,i)$-entry of $\X$ is $0$. Since $\X$ is the fundamental pattern of $M$, the $(j,i)$-entry of $\X$ is $0$ if and only if $p_i\subseteq H_j$, which, due to the inclusion-reversing property of $\phi$, holds if and only if $\phi (H_j)\subseteq\phi(p_i)$. Thus, we have that $M'\in\mqual(\X^T)$ and $r(M')=r(M)$.

Conversely, suppose $M'\in\mqual(\X^T)$ and $r(M')=r(M)$. Then $L(\X^T)\subseteq L(M')$. By \cref{thm:transpose}, we have $L(\X^T)\simeq L(\X)^*$, while $L(M)=L(\X)$ by \cref{lem:CHPIP generates lattice of flats}. Thus, we have an inclusion-reversing bijective map from $L(M)$ to $L(\X)^*$,
an inclusion-preserving lattice isomorphism from $L(\X)^*$ to $L(\X^T)$, and an inclusion-preserving injective map from $L(\X^T)$ to $L(M')$. Thus, there is an inclusion-reversing injective map from $L(M)$ to $L(M')$. Since this map involves transposing the fundamental pattern of $M$, we are guaranteed to map the hyperplanes of $M$ bijectively onto the points of $M'$.

Also, note that every row of $\X$ is nonzero and no pair of rows of $\X$ are equal. Then $M'$ is simple, since $M'\in\mqual(\X^T)$.
Thus, $M'$ is an adjoint of $M$.
\end{proof}

\begin{lemma}
    \label{lem:extend}
    Let $\X$ be a zero-nonzero pattern and $\mathcal{Y}$ be obtained by appending a column to $\X$ whose zero set is the intersection of the zero sets of some subset $A$ of the columns of $\X$.
    Suppose that $M\in\mqual(\X)$ and that $M'$ is obtained from $M$ by freely adding a point to the closure of $A$.  Then, with $E(M')$ ordered such that this new point corresponds to the new column, $M'\in\mqual(\mathcal{Y})$.
\end{lemma}

\begin{proof}
Label by $e$ the column appended to $\X$ to obtain $\mathcal{Y}$.
Recall (see \cite[Section 7.2]{o11}) that  
$M'$ is the single-element extension corresponding to $\mathcal M$,
where $\mathcal M$ is the modular cut of $M$ consisting of exactly those flats that contain $A$.  In other words, $M' = M+_{\mathcal{M}}e$.

Consider an arbitrary row $r$ of $\Y$.  Let $T$ be the zero set of the corresponding row of $\X$.  Then $T$ is a flat of $M$, since $M\in\mqual(\X)$.  Thus, $T \in \mathcal M$ if and only if $A \subseteq T$.  But $A \subseteq T$ means that row $r$ of $\X$ has a $0$ entry in every column of $A$, i.e., that $r$ is in the intersection of the zero sets of the columns of $A$.  And, by construction, this holds if and only if the zero set of row $r$ of $\Y$ is $T\cup e$.
Hence, we have shown that $T \in \mathcal M$ if and only if the zero set of row $r$ of $\Y$ is $T\cup e$.

We wish to show that $M' \in \mqual(\Y)$.  Since $r$ was chosen arbitrarily, it suffices to show that the flats of $M'$ include the zero set of row $r$ of $\Y$, which is either $T$ or $T\cup e$.  If it is $T$, then, by the above, $T\not\in \mathcal M$, so that, by \cite[Corollary 7.2.5(i)]{o11}, the set $T$ is a flat of $M'$.  If it is $T\cup e$, then $T\in\mathcal M$, so that, by \cite[Corollary 7.2.5(ii)]{o11}, the set $T\cup e$ is a flat of $M'$.
\end{proof}

Note that, in \cref{lem:extend}, the set $A$, and therefore the matroid $M'$,  is not uniquely determined by $\Y$.
In particular, if $A'\neq A$ but $Z(A')=Z(A)$, then freely adding a point to the closure of $A'$ and freely adding a point to the closure of $A$ will result in matroids in $\mqual(\mathcal{Y})$ that need not be equal.

Suppose a zero-nonzero pattern is augmented by appending a new row whose zero set is the intersection of the zero sets of some existing rows.  Since the collection of flats of a matroid is closed under intersection, this has no effect on the class of matroids associated with the pattern.  The same is not true of appending a column in the same fashion, since this increases the size of the ground set of the associated matroids.  Even so, our next result shows that there is still no effect on the matroid minimum rank.

\begin{theorem}
\label{thm:extend}
    Let $\X$ be a zero-nonzero pattern and let $\X'$ be obtained by appending a column to $\X$ whose zero set is the intersection of the zero sets of some subset of the columns of $\X$.  Then $\mr\mqual(\X')=\mr\mqual(\X)$.
\end{theorem}

\begin{proof}
Since adding a point freely to any flat of a matroid produces a new matroid with the same rank, \cref{lem:extend} shows that $\mr\mqual(\X') \le \mr\mqual(\X)$.  To prove the reverse inequality, let $M'$ be a matroid of minimum rank in $\mqual(\X')$, with an element $e$ corresponding to the column that was appended to $\X$ to produce $\X'$.  Let $M=M'\setminus e$.  Since $M'\in\mqual(\X')$, the zero set of each row of $\X'$ is a flat of $M'$.  Hence, by \cite[Proposition 3.3.7(ii)]{o11}, each row of $\X$ is a flat of $M$, so that $M\in\mqual(\X)$.  And $r(M) \le r(M')$.  Thus, $\mr\mqual(\X) \le r(M) \le r(M')=\mr\mqual(\X')$, as desired.
\end{proof}

\begin{lemma}
    \label{lem:min rank of transpose}
    Let $M$ be a matroid of rank $k$ in $\mqual(\X)$. If $M$ has an adjoint, then $\mr\mqual(\X^T)\leq k$.
\end{lemma}

\begin{proof}
    Let $M'$ be an adjoint of $M$. By definition of adjoint, $r(M')=k$. Take $\mathcal{Y}$ such that $\mathcal{Y}^T$ is the fundamental pattern of $M$. 
    Then $M\in\mqual(\Y^T)$, and by \cref{thm:adjoint and transpose of hyperplane-point pattern}, we have $M'\in\mqual(\Y)$.
    
    Now let $\intcl{\Y}$ denote the pattern obtained from $\Y$ by appending all columns (not already contained in $\Y$) whose zero sets are intersections of the zero sets of columns of $\Y$. 
By repeated use of \cref{thm:extend}, we have $\mr(\intcl{\Y})=\mr(\Y)$. Since $M'\in\mqual(\Y)$ and $r(M')=k$, we have $\mr(\intcl{\Y})=\mr(\Y)\leq k$. Therefore, there is a matroid $M''\in\mqual(\intcl{\Y})$ such that $r(M'')\leq k$. Recall that the flats of a matroid are exactly the intersections of hyperplanes of the matroid. Therefore, a set is a flat of $M$ if and only if it is the zero set of a row of $(\intcl{\Y})^T$. Thus, since $M\in\mqual(\X)$, we have that $\X$ is a row-deleted submatrix of $(\intcl{\Y})^T$. Therefore, $\X^T$ is a column-deleted submatrix of $\intcl{\Y}$. By deleting the elements of $M''$ corresponding to the columns of $\intcl{\Y}$ that are not columns of $\X^T$, we obtain a matroid in $\mqual(\X^T)$ with rank at most $k$. Therefore, $\mr\mqual(\X^T)\leq k$.
\end{proof}

\begin{theorem}\label{thm:One Big Theorem}
    Let $\X$ be a zero-nonzero pattern and let $k=\tri(\X)$.  
    If $\mr\mqual(\X) > k$, then every matroid $M\in \mqual(\X^T)$ of rank $k$ fails to have an adjoint.
\end{theorem}

\begin{proof}
We prove the contrapositive. Assume there is a matroid $M \in \mqual(\X^T)$ of rank $k$ with an adjoint $M'$. Then \cref{lem:min rank of transpose} applied to $\X^T$ implies that $\mr\mqual(\X) \leq k$.
\end{proof}

It may happen that, for a zero-nonzero pattern $\X$ with $\tri(\X)=k$, neither $\mqual(\X)$ nor $\mqual(\X^T)$ contains a matroid of rank $k$.  That is, for both the pattern and its transpose, a gap exists between the triangle number and the matroid minimum rank.  In this case, the conclusion of Theorem \ref{thm:One Big Theorem} holds vacuously. On the other hand, when a gap exists for $\X$ but not for $\X^T$, there do exist matroids of rank $k$ in $\mqual(\X^T)$, and then Theorem \ref{thm:One Big Theorem} asserts that every one of those matroids fails to have an adjoint.

As an example of the latter case, it was already noted that the pattern $\X$ of Example \ref{ex:Example 25 pattern} satisfies 
$\mr\mqual(\X) = 5 > 4=\tri(\X)$, while $\mr\mqual(\X^T)=4$.  Hence, by Theorem \ref{thm:One Big Theorem}, every matroid in $\mqual(
\X^T)$ of rank $4$ must fail to have an adjoint.
Calculations with SageMath \cite{sagemath} show that there are, up to isomorphism, $38$ such matroids, including the V\'amos matroid, which was noted to be one of them in \cref{subsec:transpose and min rank}.

\section{Uniqueness and the fundamental pattern}\label{sec:uniqueness}

For a zero-nonzero pattern $\X$ with $\tri(\X)=k$, Theorem \ref{thm:One Big Theorem} shows that
a gap between $k$ and the minimum rank of a matroid in $\mqual(\X)$ implies the nonexistence of an adjoint for every matroid  in $\mqual(\X^T)$ of rank $k=\tri(\X^T)$.  (Note that the transpose has no effect on the triangle number.)
We may therefore be interested in the special case in which 
$\mqual(\X^T)$ contains a unique such matroid.
This motivates the following question that, for simplicity, we state without the transpose. 
\begin{question}\label{question:of uniqueness}
When does a zero-nonzero pattern $\X$ have the property that $\mqual(\X)$  contains exactly one matroid of rank $\tri(\X)$?
\end{question}

The following theorem gives a partial answer to Question \ref{question:of uniqueness}.

\begin{theorem}\label{thm:uniqueness for hyperplane-point pattern}
Let $\X$ be the fundamental pattern of a matroid $M$.  Then $\mr\mqual(\X)=\tri(\X)=r(M)$ and $M$ is the unique matroid of this rank in $\mqual(\X)$.
\end{theorem}

\begin{proof}
Let $k = \tri(\X)$.
By \cref{lem:CHPIP generates lattice of flats}, we have $L(\X)=L(M)$.
Hence, $r(M)$, being equal to
the height of $L(M)$, must be equal to the height of $L(\X)$, which,
as discussed in \cref{subsec:triangle number}, is $\tri(\X)=k$.
Meanwhile, since each hyperplane of $M$ is a flat of $M$, we have $M\in\mqual(\X)$.
Hence, $\mr\mqual(\X) \le r(M)=k$.
But recall from \eqref{eqn:fundamental triangle number inequality} that $k \le \mr\mqual(\X)$ also holds.
Hence, $\mr\mqual(\X)=k=\tri(\X)=r(M)$.

Now suppose $M'\in \mqual(\X)$ has $r(M')=k$.
By the construction of $\X$ and the definition of $\mqual(\X)$, every hyperplane of $M$ is a flat of $M'$.  But every flat of $M$ is an intersection of hyperplanes of $M$,
and the flats of $M'$ are closed under intersection, so in fact every flat of $M$ (and not just every hyperplane) is a flat of $M'$.  By \cite[Proposition 7.3.6]{o11}, this implies that $M$ is a quotient of $M'$.  Since $r(M)=r(M')$, it then follows by \cite[Corollary 7.3.4]{o11} that $M=M'$.
\end{proof}

Theorem \ref{thm:uniqueness for hyperplane-point pattern} has interesting consequences in conjunction with the following prior result.

\begin{theorem}[{\cite[Corollary 29]{d20}}]
\label{thm:matrix min rank equality if representability}
Let $\mathbb F$ be an infinite field and $\X$ be a zero-nonzero pattern.  Then $\mr\mqual(\X) = \mr\qual_{\mathbb F}(\X)$ if and only if some matroid of minimum rank in $\mqual(\X)$ is representable over $\mathbb F$.
\end{theorem}

As a direct consequence of the preceding two theorems, we obtain the following.

\begin{theorem}\label{thm:uniqueness and representability}
Let $\mathbb F$ be an infinite field and $\X$ be the fundamental pattern of some matroid $M$.  Then $M$ is representable over $\mathbb F$ if and only if $\mr\qual_{\mathbb F}(\X) = \tri(\X)$.
\end{theorem}

Given an infinite field  $\F$ and a matroid $M$  not representable over $\F$, let $\X$ be the fundamental pattern of $M$. Then, by \cref{thm:uniqueness and representability}, we have $\mr\qual_{\mathbb F}(\X) > \tri(\X)$. Thus, \cref{thm:uniqueness and representability} provides an easy way to find, for an infinite field, a pattern whose matrix minimum rank exceeds its triangle number over that field. In fact, as we show in Proposition \ref{prop:gap for all fields}, we can extend this idea to all fields by choosing a matroid that is not representable over \emph{any} field.

\begin{proposition}
    \label{prop:gap for all fields}
    Let $M$ be a non-representable matroid, and let $\X$ be  its fundamental pattern. Then $\mr\qual_{\mathbb F}(\X) > \tri(\X)$ for every field $\F$.
\end{proposition}

\begin{proof}
    Let $\F$ be any field.
    By \cref{thm:uniqueness for hyperplane-point pattern}, we have $r(M)=\tri(\X)$. We know from \eqref{eqn:fundamental triangle number inequality} that $\mr\qual_{\mathbb F}(\X) \geq \tri(\X)$. Suppose for a contradiction that $\mr\qual_{\mathbb F}(\X) = \tri(\X)$. Then, by \cref{thm:matrix in Q represents matroid in R}, there is a matroid $M'\in\mqual(\X)$ that is representable over $\F$ and has rank $\tri(\X)$. But, by \cref{thm:uniqueness for hyperplane-point pattern}, this implies that $M'=M$, which contradicts the non-representability of $M$.
\end{proof}

Given any matroid $M$ that is not representable, \cref{thm:uniqueness for hyperplane-point pattern} and Proposition \ref{prop:gap for all fields} give a pattern $\X$ with $\tri(\X) = \mr\mqual(\X) < \mr\qual_\F(\X)$ over every field $\F$.  That is, this pattern has a matroid minimum rank equal to its triangle number while exhibiting, due to representability considerations, a gap between the triangle number and the matrix minimum rank over every field.  The smallest previously-known example of 
such a pattern,
namely the smallest member of the family given by \cite[Corollary 39]{d20}, has size $91\times 91$.  Here, by choosing $M$ to be, for example, the matroid $F_8$ (see \cite[p.~647]{o11}), we obtain a much smaller
example.
In particular, $F_8$ is not representable over any field, has $8$ points and $20$ hyperplanes, and therefore via Proposition \ref{prop:gap for all fields} yields a $20\times 8$ pattern with the desired property.  Similarly, using the non-Pappus matroid instead, we obtain such a pattern of size $20\times 9$.

Another application of Theorem \ref{thm:uniqueness and representability} is the following corollary.

\begin{corollary}\label{cor:represenability gap}
Let $M$ be a matroid, and let $\X$ be its fundamental pattern.
If $\mathbb F_1$ and $\mathbb F_2$ are infinite fields, with $M$ representable over $\mathbb F_1$ but not over $\mathbb F_2$, then
\[
\mr\qual_{\mathbb F_2}(\X) > \tri(\X) = \mr\qual_{\mathbb F_1}(\X).
\]
\end{corollary}

We can use Corollary \ref{cor:represenability gap} to recover some of the main results of \cite{berman_et_al}.  First, we observe that the
method given there to construct
a zero-nonzero pattern (there called the ``cycle matrix'') from a matrix representing a matroid in fact results in the fundamental pattern of the dual of that matroid. 
It is well known that a matroid is representable over a field if and only if its dual is representable over that field.
Hence, taking $M$ in \cref{cor:represenability gap} to be the dual of the matroid $AG(2,3)$ (see \cref{fig:AG23-})
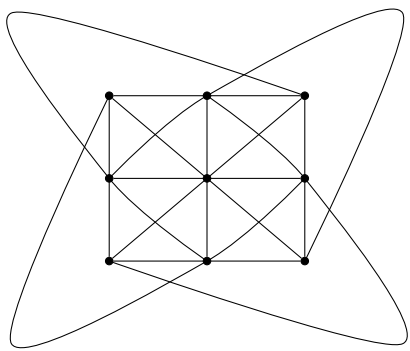
\begin{figure}
\[\begin{tikzpicture}[x=1.3cm, y=1.1cm]
	\vertex[fill,inner sep=1pt,minimum size=1pt] (1) at (0,2) [label=above:] {};
 	\vertex[fill,inner sep=1pt,minimum size=1pt] (2) at (2,2) [label=above:] {};
 	\vertex[fill,inner sep=1pt,minimum size=1pt] (3) at (0,0) [label=below:] {};
	\vertex[fill,inner sep=1pt,minimum size=1pt] (4) at (1,2) [label=below:] {};
 	\vertex[fill,inner sep=1pt,minimum size=1pt] (5) at (0,1) [label=left:] {};
 	\vertex[fill,inner sep=1pt,minimum size=1pt] (6) at (2,0) [label=below:] {};
 	\vertex[fill,inner sep=1pt,minimum size=1pt] (7) at (2,1) [label=right:] {};
 	\vertex[fill,inner sep=1pt,minimum size=1pt] (8) at (1,0) [label=above:] {};
    \vertex[fill,inner sep=1pt,minimum size=1pt] (9) at (1,1) [label=above:] {};
 	\path
 		(1) edge (3)
 		(1) edge (2)
 		(3) edge (6)
 		(2) edge (6)
            (3) edge (2)
            (5) edge (7)
            (4) edge (8)
            (1) edge (6);
\draw plot [smooth] coordinates {(0,0) (3,-1) (2,1) (1,2)};
\draw plot [smooth] coordinates {(0,1) (1,2) (3,3) (2,0)};
\draw plot [smooth] coordinates {(0,2) (-1,-1) (1,0) (2,1)};
\draw plot [smooth] coordinates {(1,0) (0,1) (-1,3) (2,2)};
\end{tikzpicture}\]
\caption{The matroid $\mathrm{AG}(2,3)$}
  \label{fig:AG23-}
\end{figure}
recovers the inequality of \cite[Corollary 2.4]{berman_et_al}, which shows that the pattern has a smaller minimum rank over $\mathbb C$ than over $\mathbb R$.  This is because $AG(2,3)$ is representable over $\mathbb C$ but not $\mathbb R$. Similarly, taking $M$ to be the dual of the Betsy Ross matroid (see \cref{fig:Betsy-Ross}) 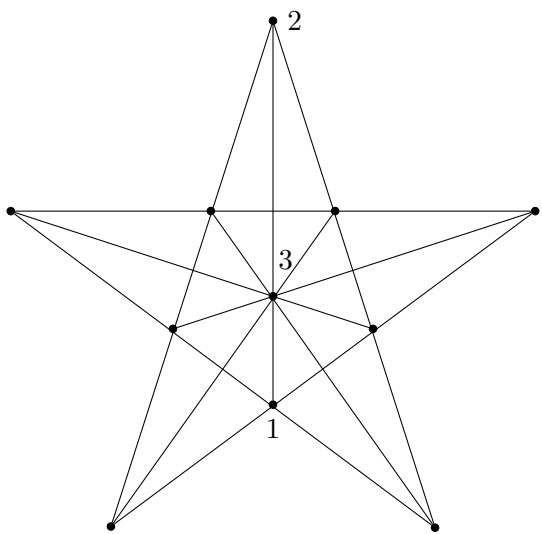
\begin{figure}
	\begin{center}
		\[\begin{tikzpicture}[x=1.4cm, y=1.4cm]
		\vertex[fill,inner sep=1pt,minimum size=1pt] (a) at (0,0) 		[label=right:$$] {};
		% What's above is there to make a visible point. What's below is there to squeeze in the label.
		\vertex[fill,inner sep=0pt,minimum size=0pt] (a) at (-0.05,.35) 		[label=right:$3$] {};
		\vertex[fill,inner sep=1pt,minimum size=1pt] (b) at (0,-1.03) 		[label=below:$1$] {};
		\vertex[fill,inner sep=1pt,minimum size=1pt] (c) at (0.95,-0.31) 	[label=-45:] {};
		\vertex[fill,inner sep=1pt,minimum size=1pt] (d) at (0.59,0.81) 	[label=45:] {};
		\vertex[fill,inner sep=1pt,minimum size=1pt] (e) at (-0.59,0.81) 	[label=135:] {};
		\vertex[fill,inner sep=1pt,minimum size=1pt] (f) at (-0.95,-0.31)	[label=225:] {};
		\vertex[fill,inner sep=1pt,minimum size=1pt] (b') at (0,2.618) 	[label=right:$2$] {};
		\vertex[fill,inner sep=1pt,minimum size=1pt] (c') at (-2.49,0.809) 	[label=above:] {};
		\vertex[fill,inner sep=1pt,minimum size=1pt] (d') at (-1.539,-2.188) 	[label=left:] {};
		\vertex[fill,inner sep=1pt,minimum size=1pt] (e') at (1.539,-2.199) 	[label=right:] {};
		\vertex[fill,inner sep=1pt,minimum size=1pt] (f') at (2.490,0.809) 	[label=above:] {};
		\path
		(b) edge (b')
		(c) edge (c')
		(d) edge (d')
		(e) edge (e')
		(f) edge (f')
		(d') edge (b')
		(b') edge (e')
		(e') edge (c')
		(f') edge (d')
		(c') edge (f');
		\end{tikzpicture}\]
		\caption{The Betsy Ross matroid}
		\label{fig:Betsy-Ross}
		\end{center}
\end{figure} yields the inequality of \cite[Corollary 2.7]{berman_et_al}, which shows that the pattern has a smaller minimum rank over $\mathbb R$ than over $\mathbb Q$. This is because the Betsy Ross matroid is representable over $\mathbb R$ but not $\mathbb Q$.

The patterns considered above are somewhat large, with sizes $66\times9$ and $170\times11$, respectively. These sizes result from the fact that $AG(2,3)$ and the Betsy Ross matroid have $9$ and $11$ points, respectively, and $66$ and $170$ circuits, respectively. (Details of these counts are given in \cite{berman_et_al}.) 
This implies that their dual matroids have $66$ and $170$ hyperplanes, respectively.  But both $AG(2,3)$ and the Betsy Ross matroid have smaller numbers of hyperplanes than their duals, and therefore we may obtain smaller patterns that serve as examples with the same properties by applying \cref{cor:represenability gap} directly to $AG(2,3)$ and the Betsy Ross matroid rather than their dual matroids. There are $12$ hyperplanes in $AG(2,3)$,
namely each of its three-point lines.
Therefore, \cref{cor:represenability gap} gives a $12\times9$ pattern $\X$ with $\mr\qual_{\mathbb R}(\X) > \mr\qual_{\mathbb C}(\X)$. There are $20$ hyperplanes in the Betsy Ross matroid. (These hyperplanes comprise five lines with four points, five lines with three points, and ten lines with two points.) Therefore, \cref{cor:represenability gap} gives a $20\times11$ pattern $\X$ with $\mr\qual_{\mathbb Q}(\X) >  \mr\qual_{\mathbb R}(\X)$.

In fact, we can make the patterns even smaller. The matroid $AG(2,3)$ is highly symmetric, in that the deletion of any one of its nine points results in the same matroid, up to isomorphism. Call this matroid $AG(2,3)\backslash e$. Since every hyperplane of $AG(2,3)$ has three points, deleting an element does not change the number of hyperplanes. Thus, $AG(2,3)\backslash e$ still has $12$ hyperplanes (eight lines with three points and four lines with two points). It is fairly well known (and not difficult to show; see \cite[Example 6.11 and Exercise 6.29]{gm12}) that $AG(2,3)\backslash e$ is representable over $\mathbb{C}$ but not $\mathbb{R}$. Therefore, \cref{cor:represenability gap} with $M=AG(2,3)\backslash e$ gives a $12\times8$ pattern $\X$ with $\mr\qual_{\mathbb R}(\X) >  \mr\qual_{\mathbb C}(\X)$.

By deleting points $1$ and $2$ from the Betsy Ross matroid (as labeled in \cref{fig:Betsy-Ross}), we obtain what we will call the Perles matroid, whose geometric representation is the Perles configuration. This point-and-line configuration was introduced by Micha Perles, who showed that it cannot be realized in the plane with rational coordinates \cite{zie08}.  This implies that the Perles matroid is representable over $\mathbb{R}$ but not $\mathbb{Q}$ (just like the Betsy Ross matroid). The lines in the Betsy Ross matroid that still have at least two points after points $1$ and $2$ are deleted are the hyperplanes of the Perles matroid. Each of the points $1$ and $2$ is in two lines that contain only two points. These, in addition to the line $\{1,2,3\}$, are the five hyperplanes that do not persist after points $1$ and $2$ are deleted. Thus, the Perles matroid has 15 hyperplanes, and \cref{cor:represenability gap}, applied with $M$ taken to be the Perles matroid, gives a $15\times9$ pattern $\X$ with $\mr\qual_{\mathbb Q}(\X) >  \mr\qual_{\mathbb R}(\X)$.

Patterns with the above properties that are even smaller still are obtained in \cite{d20} from $AG(2,3)\backslash e$ and the Perles matroid by applying \cite[Theorem 36]{d20}.
That theorem applies only to matroids of rank $3$ satisfying certain combinatorial conditions, and allows for only the rows corresponding to the \emph{dependent} hyperplanes to be included in the construction of the pattern, while guaranteeing the same uniqueness provided by \cref{thm:uniqueness for hyperplane-point pattern}.  The question of whether (and how) this could be done in general is one we highlight in \cref{sec:questions and future directions}.

In the remainder of this section, we establish a partial converse to \cref{thm:One Big Theorem}, in the setting of the uniqueness given by \cref{thm:uniqueness for hyperplane-point pattern}.

As a motivating example, consider the V\'amos matroid $V$.  Recall that $V$ has rank $4$ and no adjoint.  Let $\X$ be such that $\X^T$ is the fundamental pattern of $V$.  By \cref{thm:uniqueness for hyperplane-point pattern}, the class $\mr\mqual(\X^T)$ contains exactly one matroid of rank $4=\tri(\X^T)$, namely $V$.  Hence, by \cref{thm:adjoint and transpose of hyperplane-point pattern}, any matroid of rank $4$ in $\mqual(\X)$ would be an adjoint of $V$.  But such an adjoint does not exist, and therefore $\mr\mqual(\X) > 4 = \tri(\X)$.

The following result generalizes the above example, giving a partial converse to \cref{thm:One Big Theorem}.
Recall that, for a zero-nonzero pattern $\X$ with $\tri(\X)=k$, we also have $\tri(\X^T)=k$, so that
$k$ is a lower bound on the rank of every matroid occurring in either $\mqual(\X)$ or $\mqual(\X^T)$.
\cref{thm:One Big Theorem} showed that if there is
a gap between $k$ and the minimum rank of a matroid in $\mqual(\X)$, then every matroid of rank $k$
in $\mqual(\X^T)$ fails to have an adjoint.
But there is exactly
one matroid $M$ of rank $k$ in $\mqual(\X^T)$
 when $\X^T$ is the fundamental pattern of $M$.  This is the uniqueness given by \cref{thm:uniqueness for hyperplane-point pattern}.
In this setting,
we have a converse to \cref{thm:One Big Theorem}:\ If $M$ fails to have an adjoint, then there is a gap between $k$ and the minimum rank of a matroid in $\mqual(\X)$.  That is the content of the following theorem.

\begin{theorem}
    Let $M$ be a matroid and let $\X$ be a zero-nonzero pattern such that $\X^T$ is the fundamental pattern of $M$.  Then $M$ has no adjoint if and only if $\mr\mqual(\X) > \tri(\X)$.
\end{theorem}

\begin{proof}
First, suppose $M$ has an adjoint $M'$. By \cref{thm:uniqueness for hyperplane-point pattern}, $\mr\mqual(\X^T)=\tri(\X^T)=r(M)$. Since a pattern and its transpose have the same triangle number, $\tri(\X)=\tri(\X^T)=r(M)=r(M')$. Recall from \eqref{eqn:fundamental triangle number inequality} that $\mr\mqual(\X)\geq\tri(\X)$ for every pattern $\X$. By \cref{thm:adjoint and transpose of hyperplane-point pattern}, we have $M'\in\mqual(\X)$. Therefore, since $r(M')=\tri(\X)$, we conclude that $\mr\mqual(\X)=\tri(\X)$.

Conversely, assume $M$ has no adjoint.  Then \cref{thm:adjoint and transpose of hyperplane-point pattern} applied to $\X^T$ implies that no matroid in $\mqual(\X)$ can have rank equal to $r(M)$. Thus, $\mr\mqual(\X) > r(M)$ and \cref{thm:uniqueness for hyperplane-point pattern} implies that $r(M)=\tri(\X^T)=\tri(\X)$.    
\end{proof}

\section{Questions and future directions}\label{sec:questions and future directions}

\cref{thm:One Big Theorem} states that, for a pattern $\X$ with $\tri(\X)=k$, a gap of $\mr\mqual(\X) > k$ implies that every matroid of rank $k$ in $\mqual(\X^T)$ fails to have an adjoint.
The naive converse would state that when every matroid of rank $k$ in the class $\mqual(\X^T)$ fails to have an adjoint, $\mr\mqual(\X) > k$.
One complication, however, is the possibility that this holds vacuously in that  $\mqual(\X^T)$ simply does not contain any matroids of rank $k$.  

As an example of this phenomenon, let $\Y$ be the pattern given in \eqref{eqn:ex 25 transpose}.  Then $\Y^T$ is the pattern of \cref{ex:Example 25 pattern}.  
As observed in that example,
this pattern has $\tri(\Y^T)=4$, but $\mqual(\Y^T)$ does not contain any matroids of rank $4$.  Vacuously, then, every matroid of rank $4$ in $\mqual(\Y^T)$ fails to have an adjoint.  Even so, we have that $\mr\mqual(\Y)=4=\tri(\Y)$, as observed at the start of \cref{subsec:transpose and min rank}.

Hence, the ``naive converse'' to \cref{thm:One Big Theorem} described above fails to hold.
In seeking a converse 
that avoids this complication, it is natural to consider the following.

\begin{question}
Let $k=\tri(\X)$ and suppose the class $\mqual(\X^T)$ contains at least one matroid of rank $k$.  If every such matroid fails to have an adjoint, does it follow that $\mr\mqual(\X) > k$?
\end{question}

Another goal for future research would be to establish the uniqueness 
of \cref{thm:uniqueness for hyperplane-point pattern} for more than just the fundamental pattern $\X$ of a matroid $M$.
In particular,
a natural question is whether all of the rows of $\X$ are really necessary.
Specifically, it seems that often some rows can be deleted to give a smaller pattern $\X'$ with the same triangle number such that $M$ remains the unique matroid of smallest rank in $\mqual(\X')$.
Along these lines, \cite[Theorem 36]{d20} shows that when $M$ is a matroid of rank $3$ meeting certain combinatorial conditions, it suffices to keep only the rows of $\X$ corresponding to the \emph{dependent} hyperplanes of $M$.  But the question of what in general characterizes a collection of hyperplanes giving this uniqueness seems ripe for exploration.

\begin{question}
\label{que:how many hyperplanes}
Let $M$ be a matroid and let $\X$ be its fundamental pattern.  What conditions on a collection $\mathcal C$ of hyperplanes of $M$ are sufficient to imply that when only the rows of $\X$ corresponding to those hyperplanes in $\mathcal C$ are retained (meaning that all others are deleted), the result is a zero-nonzero pattern $\X'$ with the property that $M$ is the unique matroid of minimum rank in $\mqual(\X')$?
\end{question}

The above question leads to the following purely matroid-theoretic problem.

\begin{question}\label{q:minimal hyperplane collection for uniqueness in strictly matroid terms}
Let $M$ be a matroid of rank $k$.  Determine a minimal collection $\mathcal C$ of hyperplanes of $M$ with the property that whenever $M'$ is a matroid of rank $k$ with $E(M')=E(M)$ such that every set in $\mathcal C$ is a flat of $M'$, then in fact $M'=M$.
\end{question}

Note that in \cref{q:minimal hyperplane collection for uniqueness in strictly matroid terms}, replacing the condition that each set in $\mathcal C$ is a flat of $M'$ with the stronger condition that each such set is actually a hyperplane of $M'$ results in a problem dual (and hence equivalent) to that of determining a minimal set of circuits that, along with the rank, are sufficient to determine the matroid. In fact, problems closely related to this have already received some attention (in, e.g., \cite{m2014,oxley_semple_whittle,md2017}).

A better understanding of Questions \ref{que:how many hyperplanes} and \ref{q:minimal hyperplane collection for uniqueness in strictly matroid terms} seems likely to reveal that in many cases a matroid (of a given rank) is determined by a relatively small collection of its hyperplanes.
This would allow smaller patterns to suffice for examples like those in Section \ref{sec:uniqueness}, where properties (such as representability) of the matroid translate into desired properties of the resulting pattern.

\section*{Acknowledgments}

In support of the research presented here, the first author received grant-in-aid funding from the  College of Arts \& Sciences of Quinnipiac University. The second author was supported in part by an AMS-Simons Travel Grant.
The authors are indebted to
H.~Tracy Hall for the suggestion to consider the notion of a matroid adjoint as relevant to the matroid minimum rank problem, 
made to the first author following the workshop \emph{Zero forcing and its applications} at the American Institute of Mathematics in 2017.
Python code for the exhaustive search of matroids of a fixed rank within a given pattern class,  referenced at the end of \cref{sec:transposes and adjoints}, was written by Aiden Rosen as part of the capstone project he completed for his computer science major at Quinnipiac University.

\bibliographystyle{alpha}
\bibliography{main.bib}

\end{document}